\documentclass[10pt]{article}

\usepackage{amssymb,latexsym,amsmath,epsfig,amsthm, tikz} 
\usepackage{xcolor,colortbl}
\makeatletter

\newtheorem{theorem}{Theorem}

\theoremstyle{definition}

\begin{document}

\begin{center}
\uppercase{\bf \boldmath The $3\times 3$ Solution to the Game of Tak}
\vskip 20pt
{\bf Joseph E. Marrow}\\
{\it Department of Mathematics, Southern Utah University, Cedar City, UT, USA}\\
{\tt josephmarrow@suu.edu}\\
\end{center}
\vskip 20pt

\centerline{\bf Abstract}

\noindent
We provide answers to natural combinatorial questions which arise from learning the rules to Tak. We solve the discarded cases of the $3\times 3$ game not previously considered by Joshua Achiam, showing that Player $1$ has a forced win from any starting position. Diagrams for both his and our solutions are provided in the hopes of making the results more accessible. 

\section{Introduction}
Tak is a strategy game first described in a work of fiction \cite{wiseman} which was later made into a playable game in \cite{takrules}. Some study has been done into variations of the game \cite{math}; a considerable amount not published, but kept on various forums. We provide some original results, and answer some natural combinatorial questions which directly follow from the rules of the game. In the Appendix we also type up Joshua Achiam's solution to the $3\times 3$ game, providing various diagrams to help the reader follow. He omits fours cases of initial position; which, while not needed for $W$ to ensure victory, the author feels should be included for the sake of completeness. These generalize the $3\times 3$ solution to show that $W$ can secure victory independent of the initial positions of the stones.

\begin{figure}[h!]
\label{setup}
\caption{}
\centering
\begin{tikzpicture}
    \foreach \x in {0,1,2,3,4,5} {
        \draw (\x,0) -- (\x,5); 
        \draw (0,\x) -- (5,\x); 
    }
    \filldraw (2.5, 6) circle (2pt) node[above] {$B$};

    \draw (2.5, -1) circle (2pt) node[below] {$W$};
    \node[above, rotate=180] at (2.5, 5.5) {North}; 
    \node[below] at (2.5, 0) {South}; 
    \node[rotate=90] at (5.5, 2.5) {East}; 
    \node[rotate=-90] at (-0.5, 2.5) {West};

    \foreach \y in {1,2,3,4,5} {
        \node at (0.2, \y - 0.5) {\y};
    }

    \foreach \x [count=\i] in {a,b,c,d,e} {
        \node at (\i - 0.5, 0.2) {\x};
    }
    \end{tikzpicture}
\end{figure}

\section{Rules}\label{rules}
Many explanations of the rules can be found \cite{takrules, master1, master2}, but we explain the fundamental gameplay here. An $n\times n$ board is placed between two players, one playing white ($W$) and the other playing black ($B$). The edge of the board closest to white is the south side, with the remaining sides labeled accordingly.

Each player has a number of appropriately colored stones, determined by the size of the board, see Figure \ref{stonecount}. These stones can either be played horizontally (road stones, also called flat stones) or vertically (standing stones, also called walls). The goal of the game is to be the first player to create a road of a single color stones using orthogonally connected squares which joins opposite edges (north -- south, east -- west). Standing stones do not count as part of the road. The game begins with $W$ placing one of blacks road stones, $B$ placing one of $W$'s road stones, and then players alternate placing or moving one of their own stones\footnote{Sequential games often alternate who plays first, throughout we will suppose it is White.}.

\begin{table}
\label{stonecount}
\begin{center}
\begin{tabular}{ccc}
Board Size & Number of stones & Capstones\\
\hline
3 & 10 & 0\\
4 & 15 & 0\\
5 & 21 & 1\\
6 & 30 & 1\\
7 & 40 & 2\\ 
8 & 50 & 2
\end{tabular}
\end{center}
\end{table}

Players may move stones orthogonally in any available direction. Stones (road or standing) may move onto road stones, stacking them into a tower. Color does not matter when stacking stones. There is no limit to how high a stack of stones can become, excepting the limit of pieces available. Stones cannot be stacked upon a standing stone. Stones may only be placed in empty squares. If there are no empty squares, or if one player is out of stones, the game ends. In this case victory is determined by whoever has the most flat stones controlling squares\footnote{Ties are possible if the number of flat stones is equal, but this is generally considered uncommon. While not an official rule, most tournament play also requires a new stone to be dropped at least once every $30$ moves or the game will be drawn. This is analogous to the $50$-move rule in chess, and forces every game to eventually end.}.

There is a special stone, called a capstone, which can be moved on top of a standing stone. When this happens, the standing strong is crushed (flattened) into a flat stone. No piece can stack on top of a capstone, but the capstone does count as part of the road.

The player whose stone is on top of a tower controls the tower. When moving a tower, the player who controls the tower may take up to $n$ stones (the carry limit) from the top, and move in one direction, leaving at least one stone from the bottom of the stack behind for each square moved through.

So on a $4\times 4$ grid, if $B$ chooses to move the top four stones in the tower shown below, there are $8$ options for how the tower could fall. In the following diagram, the tower is falling to the right.

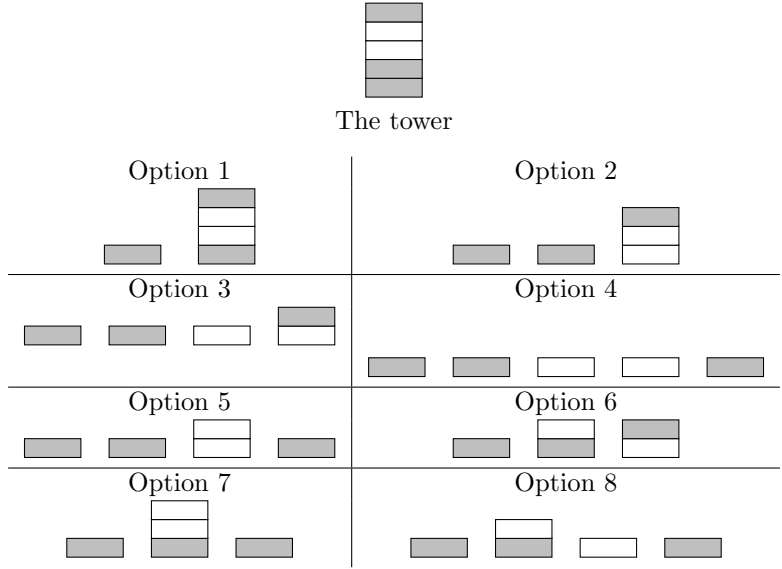
\begin{center}
\begin{tikzpicture}
    \def\width{0.75}
    \def\height{0.25}

    \filldraw[fill=lightgray, draw=black] (0, 0) rectangle (\width, \height);
    \filldraw[fill=lightgray, draw=black] (0, \height) rectangle (\width, 2 * \height);
    \filldraw[fill=white, draw=black] (0, 2 * \height) rectangle (\width, 3 * \height);
    \filldraw[fill=white, draw=black] (0, 3 * \height) rectangle (\width, 4 * \height);
    \filldraw[fill=lightgray, draw=black] (0, 4* \height) rectangle (\width, 5 * \height);
\end{tikzpicture}\\
The tower
\end{center}

\begin{center}
\begin{tabular}{c|c}
Option 1 & Option 2\\
\begin{tikzpicture}
    \def\width{0.75}
    \def\height{0.25}

    \filldraw[fill=lightgray, draw=black] (0, 0) rectangle (\width, \height);
    \filldraw[fill=lightgray, draw=black] (\width+0.5, 0) rectangle (2*\width +0.5, 1 * \height);
    \filldraw[fill=white, draw=black] (\width+0.5, 1 * \height) rectangle (2*\width +0.5, 2 * \height);
    \filldraw[fill=white, draw=black] (\width+0.5, 2 * \height) rectangle (2*\width +0.5, 3 * \height);
    \filldraw[fill=lightgray, draw=black] (\width+0.5, 3* \height) rectangle (2*\width +0.5, 4 * \height);
\end{tikzpicture}
&
\begin{tikzpicture}
    \def\width{0.75}
    \def\height{0.25}

    \filldraw[fill=lightgray, draw=black] (0, 0) rectangle (\width, \height);
    \filldraw[fill=lightgray, draw=black] (1.5*\width, 0) rectangle (2.5*\width, 1 * \height);
    \filldraw[fill=white, draw=black] (3*\width, 0) rectangle (4*\width, 1 * \height);
    \filldraw[fill=white, draw=black] (3*\width, 1 * \height) rectangle (4*\width, 2 * \height);
    \filldraw[fill=lightgray, draw=black] (3*\width, 2* \height) rectangle (4*\width, 3 * \height);
\end{tikzpicture}
\\
\hline
Option 3 & Option 4\\
\begin{tikzpicture}
    \def\width{0.75}
    \def\height{0.25}

    \filldraw[fill=lightgray, draw=black] (0, 0) rectangle (\width, \height);
    \filldraw[fill=lightgray, draw=black] (1.5*\width, 0) rectangle (2.5*\width , \height);
    \filldraw[fill=white, draw=black] (3*\width, 0) rectangle (4*\width, \height);
    \filldraw[fill=white, draw=black] (4.5*\width,  0) rectangle (5.5*\width, \height);
    \filldraw[fill=lightgray, draw=black] (4.5*\width, \height) rectangle (5.5*\width, 2 * \height);
\end{tikzpicture}\\
&
\begin{tikzpicture}
    \def\width{0.75}
    \def\height{0.25}

    \filldraw[fill=lightgray, draw=black] (0, 0) rectangle (\width, \height);
    \filldraw[fill=lightgray, draw=black] (1.5*\width, 0) rectangle (2.5*\width , \height);
    \filldraw[fill=white, draw=black] (3*\width, 0) rectangle (4*\width, \height);
    \filldraw[fill=white, draw=black] (4.5*\width, 0) rectangle (5.5*\width, \height);
    \filldraw[fill=lightgray, draw=black] (6*\width, 0) rectangle (7*\width, \height);
\end{tikzpicture}
\\
\hline
Option 5 & Option 6\\
\begin{tikzpicture}
    \def\width{0.75}
    \def\height{0.25}

    \filldraw[fill=lightgray, draw=black] (0, 0) rectangle (\width, \height);
    \filldraw[fill=lightgray, draw=black] (1.5*\width, 0) rectangle (2.5*\width, 1 * \height);
    \filldraw[fill=white, draw=black] (3*\width, 0) rectangle (4*\width, 1 * \height);
    \filldraw[fill=white, draw=black] (3*\width, 1 * \height) rectangle (4*\width, 2 * \height);
    \filldraw[fill=lightgray, draw=black] (4.5*\width, 0) rectangle (5.5*\width,  \height);
\end{tikzpicture}
&
\begin{tikzpicture}
    \def\width{0.75}
    \def\height{0.25}

    \filldraw[fill=lightgray, draw=black] (0, 0) rectangle (\width, \height);
    \filldraw[fill=lightgray, draw=black] (1.5*\width, 0) rectangle (2.5*\width, 1 * \height);
    \filldraw[fill=white, draw=black] (1.5*\width, \height) rectangle (2.5*\width, 2 * \height);
    \filldraw[fill=white, draw=black] (3*\width, 0) rectangle (4*\width, \height);
    \filldraw[fill=lightgray, draw=black] (3*\width, \height) rectangle (4*\width, 2 * \height);
\end{tikzpicture}\\
\hline
Option 7 & Option 8\\
\begin{tikzpicture}
    \def\width{0.75}
    \def\height{0.25}

    \filldraw[fill=lightgray, draw=black] (0, 0) rectangle (\width, \height);
    \filldraw[fill=lightgray, draw=black] (1.5*\width, 0) rectangle (2.5*\width, 1 * \height);
    \filldraw[fill=white, draw=black] (1.5*\width, 1 * \height) rectangle (2.5*\width, 2 * \height);
    \filldraw[fill=white, draw=black] (1.5*\width, 2 * \height) rectangle (2.5*\width, 3 * \height);
    \filldraw[fill=lightgray, draw=black] (3*\width, 0) rectangle (4*\width, \height);
\end{tikzpicture}
&
\begin{tikzpicture}
    \def\width{0.75}
    \def\height{0.25}

    \filldraw[fill=lightgray, draw=black] (0, 0) rectangle (\width, \height);
    \filldraw[fill=lightgray, draw=black] (1.5*\width, 0) rectangle (2.5*\width, 1 * \height);
    \filldraw[fill=white, draw=black] (1.5*\width, \height) rectangle (2.5*\width, 2 * \height);
    \filldraw[fill=white, draw=black] (3*\width, 0) rectangle (4*\width, \height);
    \filldraw[fill=lightgray, draw=black] (4.5*\width, 0) rectangle (5.5*\width, \height);
\end{tikzpicture}
\end{tabular}
\end{center}

Now, $B$ need not remove all four, but if they choose to do so, there are only these $8$ ways for the tower to fall. Since this tower is larger than the carry limit, it is called a tall tower.

\begin{theorem}
There are $2^{\min(n,t)}-1$ possible ways a tower with height $t$ can fall.
\end{theorem}
\begin{proof} Considering the sum of compositions of integers \cite[A000079]{oeis} we find
$$
\sum_{k=1}^{\min(n, t)} 2^{k-1} = 2^{\min(n,t)}-1.
$$
\end{proof}

\begin{theorem}
If a stone must be placed every $k$ moves, then the maximum length of a game is bounded by $2f(n)k$.
\end{theorem}
\begin{proof}
Let $f(n)$ be the function which takes in the board size, and gives the total number of available stones for each player. That is $f(3)=10$, $f(4)=15$, $f(5)=22$, $f(6)=31$, $f(7)=42$, and $f(8)=52$. 
Then $2f(n)$ is the number of stones available, and $k$ the number of moves between each stone being placed. Be warned that this is not a sharp bound, as the game ends when one player has played their final stone, not when they both do.
\end{proof}

\section{Notation}\label{notation}
We will label the spaces of the board similar to a chessboard, with letters along the columns and numbers along the rows, see the initial figure. We notate playing a stone by writing the square where it is placed. If it is a standing stone this is first denoted with an $S$, if it is the capstone with a $C$. 

When moving pieces we write the square it begins on. We use $>, <, +, -$ to denote which direction the piece moves in. If the number increases (N) we use $+$, if it decreases (S) we use $-$. If the letter increases (E) we use $>$, whereas $<$ is for decreasing letters (W). Some prefer $\uparrow$, $\downarrow$, $\leftarrow$, $\rightarrow$ in place of $+, -, <, >$.

If a tower is moving, after the direction indicator we write a sequence of numbers denoting how many stones are left behind in order. So $c2>231$ takes the top six stones from $c2$, puts the bottom $2$ on $d2$, the next three on $e2$, and the final topmost stone on $f2$. Such a move would only be possible on a $6\times 6$ grid or larger for two reasons. First, $f2$ needs to exist, and second $2+3+1=6$ stones were moved\footnote{Some players would prefer to write this move $6c2>231$, where the initial $6$ indicates that six stones were moved. We omit this as it's implied by the count of stones dropped. Likewise some players use an initial $F$ to denote a flat stone, but as this is the most common we leave it unstated.}.

\section{Solution when $n=3$}\label{aturan}
The smallest nontrivial game is when $n=3$, but White has a forced win on a $3\times 3$ board, as shown by \cite{solved}, see the Appendix. We modify the rules slightly to an Aturan start \cite{takrules}, and consider if victory is determinable. We devote the remainder of the paper to showing this. 

\begin{theorem}
On a $3\times 3$ board, White has a forced win from any initial position.
\end{theorem}

Since both initial pieces are placed randomly it does not matter who places which stones in this case, so we write it as White placing their own stone. We must add the following positions to consideration. Omitted positions are equivalent to states considered here or in the Appendix, although equivalent states reached during play are not always removed from the diagrams. Of note, if Black begins in a corner, then it does not matter where White lands, this was considered in the Appendix. 

\begin{center}
\begin{tabular}{cc}
Game 1 & Game 2\\\\
\begin{tikzpicture}
    \foreach \x in {0,1,2,3} {
        \draw (\x,0) -- (\x,3); 
        \draw (0,\x) -- (3,\x); 
    }
    

    \filldraw[fill=white, draw=black] (0.25, 0.25) rectangle (0.75, 0.75);
    \filldraw[fill=white, draw=black] (0.25, 1.25) rectangle (0.75, 1.75);
    \filldraw[fill=lightgray, draw=black] (1.25, 0.25) rectangle (1.75, 0.75);
\end{tikzpicture}
&
\begin{tikzpicture}
    \foreach \x in {0,1,2,3} {
        \draw (\x,0) -- (\x,3); 
        \draw (0,\x) -- (3,\x); 
    }
    

    \filldraw[fill=lightgray, draw=black] (1.25, 1.25) rectangle (1.75, 1.75);
    \filldraw[fill=white, draw=black] (1.25, 2.25) rectangle (1.75, 2.75);
    \filldraw[fill=white, draw=black] (0.25, 2.25) rectangle (0.75, 2.75);
\end{tikzpicture}\\
1. a1 b1, 2. a2 & 1. a3 b2, 2. b3\\
1. a2 b2, 2. a1 & 1. b3 b2, 2. a3\\
\\
Game 3 & Game 4\\
\\
\begin{tikzpicture}
    \foreach \x in {0,1,2,3} {
        \draw (\x,0) -- (\x,3); 
        \draw (0,\x) -- (3,\x); 
    }
    

    \filldraw[fill=lightgray, draw=black] (1.25, 0.25) rectangle (1.75, 0.75);
    \filldraw[fill=white, draw=black] (2.25, 2.25) rectangle (2.75, 2.75);
    \filldraw[fill=white, draw=black] (1.25, 2.25) rectangle (1.75, 2.75);
\end{tikzpicture}
 & 
\begin{tikzpicture}
    \foreach \x in {0,1,2,3} {
        \draw (\x,0) -- (\x,3); 
        \draw (0,\x) -- (3,\x); 
    }
    

    \filldraw[fill=white, draw=black] (1.25, 1.25) rectangle (1.75, 1.75);
    \filldraw[fill=lightgray, draw=black] (1.25, 0.25) rectangle (1.75, 0.75);
    \filldraw[fill=white, draw=black] (0.25, 1.25) rectangle (0.75, 1.75);
\end{tikzpicture}
\\
1. c3 b1, 2. b3 & 1. b2 b1 2. a2\\
1. b3 b1, 2. c3 & \\
\end{tabular}
\end{center}

After Game $1$ is reached, Black has three possible moves, namely $a3$, $Sa3$, and $b1<$. In the diagrams which follow, the columns are shaded according to who plays the move. We provide Whites responses, and consider each of Blacks moves which do not immediately give White a winning move. We have attempted to remove equivalent positions from the chart. We use parenthesis to indicate an optional distinction; i.e. $(S)a1$ denotes either $Sa1$ or $a1$, and $b2<(2)$ means either $b2<1$ or $b2<2$. Empty spaces to the left of moves in the diagram are understood to be the same as the first move above the empty space.

\begin{center}
\begin{tikzpicture}
\node (G1) at (3,1) {Game 1};
\filldraw[fill=lightgray, draw=none] (-0.5, 0.5) rectangle (0.5,-17);
\filldraw[fill=lightgray, draw=none] (1.5, 0.5) rectangle (2.5,-17);
\filldraw[fill=lightgray, draw=none] (3.5, 0.5) rectangle (4.5,-17);
\filldraw[fill=lightgray, draw=none] (5.5, 0.5) rectangle (6.5,-17);
\node (r1c0) at (0,0) {$a3$};
\node (r1c1) at (1,0) {$b3$};
\node (r1c2) at (2,0) {$b1<1$};
\node (r1c3) at (3,0) {$b2$};
\node (r2c2) at (2,-0.75) {$a3-1$};
\node (r2c3) at (3,-0.75) {$a3$};
\node (r2c4) at (4,-0.75) {$(S)c3$};
\node (r2c5) at (5,-0.75) {$a1+1$};
\node (r3c0) at (0,-1.5) {$Sa3$};
\node (r3c1) at (1,-1.5) {$a1>1$};
\node (r3c2) at (2,-1.5) {$(S)a1$};
\node (r3c3) at (3,-1.5) {$b2$};
\node (r4c2) at (2,-2.25) {$(S)c1$};
\node (r4c3) at (3,-2.25) {$b2$};
\node (r5c2) at (2,-3) {$(S)c3$};
\node (r5c3) at (3,-3) {$b2$};
\node (r6c2) at (2,-3.75) {$b2$};
\node (r6c3) at (3,-3.75) {$a2>1$};
\node (r6c4) at (4,-3.75) {$(S)b3$};
\node (r6c5) at (5,-3.75) {$c1$};
\node (r7c4) at (4,-4.5) {$a3>1$};
\node (r7c5) at (5,-4.5) {$c1$};
\node (r8c2) at (2,-5.25) {$(S)b3$};
\node (r8c3) at (3,-5.25) {$c1$};
\node (r9c2) at (2,-6) {$(S)c2$};
\node (r9c3) at (3,-6) {$c1$};
\node (r9c4) at (4,-6) {$c2-1$};
\node (r9c5) at (5,-6) {$b2$};
\node (r10c2) at (2,-6.75) {$a3>1$};
\node (r10c3) at (3,-6.75) {$a1$};
\node (r11c2) at (2,-7.5) {$Sb2$};
\node (r11c3) at (3,-7.5) {$c1$};
\node (r11c4) at (4,-7.5) {$Sa1$};
\node (r11c5) at (5,-7.5) {$c2$};
\node (r12c0) at (0,-8.25) {$b1<1$};
\node (r12c1) at (1,-8.25) {$a2-$};
\node (r12c2) at (2,-8.25) {$Sa2$};
\node (r12c3) at (3,-8.25) {$a1>12$};
\node (r12c4) at (4,-8.25) {$(S)a1$};
\node (r12c5) at (5,-8.25) {$b3$};
\node (r12c6) at (6,-8.25) {$(S)b2$};
\node (r12c7) at (7,-8.25) {$c2$};
\node (r13c4) at (4,-9) {$a2-1$};
\node (r13c5) at (5,-9) {$b2$};
\node (r13c6) at (6,-9) {$a1>1$};
\node (r13c7) at (7,-9) {$c2$};
\node (r14c2) at (2,-9.75) {$a2$};
\node (r14c3) at (3,-9.75) {$a1+12$};
\node (r14c4) at (4,-9.75) {$(S)a1$};
\node (r14c5) at (5,-9.75) {$b2$};
\node (r14c6) at (6,-9.75) {$a1+1$};
\node (r14c7) at (7,-9.75) {$b3$};
\node (r15c2) at (2,-10.5) {$Sa3$};
\node (r15c3) at (3,-10.5) {$a1>12$};
\node (r15c4) at (4,-10.5) {$(S)a1$};
\node (r15c5) at (5,-10.5) {$b2$};
\node (r15c6) at (6,-10.5) {$a1>1$};
\node (r15c7) at (7,-10.5) {$c2$};
\node (r16c2) at (2,-11.25) {$a3$};
\node (r16c3) at (3,-11.25) {$b1$};
\node (r16c4) at (4,-11.25) {$(S)c1$};
\node (r16c5) at (5,-11.25) {$a1+12$};
\node (r17c2) at (2,-12) {$(S)b2$};
\node (r17c3) at (3,-12) {$c3$};
\node (r17c4) at (4,-12) {$(S)b3$};
\node (r17c5) at (5,-12) {$c1$};
\node (r18c4) at (4,-12.75) {$(S)a3$};
\node (r18c5) at (5,-12.75) {$c1$};
\node (r19c4) at (4,-13.5) {$(S)a2$};
\node (r19c5) at (5,-13.5) {$c1$};
\node (r20c4) at (4,-14.25) {$b2+1$};
\node (r20c5) at (5,-14.25) {$c1$};
\node (r21c4) at (4,-15) {$b2<1$};
\node (r21c5) at (5,-15) {$c1$};
\node (r22c2) at (2,-15.75) {$(S)b3$};
\node (r22c3) at (3,-15.75) {$a2$};
\node (r22c4) at (4,-15.75) {$b3<1$};
\node (r22c5) at (5,-15.75) {$a1>12$};
\node (r23c4) at (4,-16.5) {$(S)a3$};
\node (r23c5) at (5,-16.5) {$a1>12$};
\end{tikzpicture}
\end{center}

\begin{center}
\begin{tikzpicture}
\filldraw[fill=lightgray, draw=none] (0.5, -0.5) rectangle (1.5,-16.5);
\filldraw[fill=lightgray, draw=none] (2.5, -0.5) rectangle (3.5,-15.5);
\filldraw[fill=lightgray, draw=none] (4.5, -0.5) rectangle (5.5,-16.5);
\filldraw[fill=lightgray, draw=none] (6.5, -0.5) rectangle (7.5,-16.5);
\filldraw[fill=lightgray, draw=none] (9.5, -0.5) rectangle (10.5,-16.5);
\filldraw[fill=lightgray, draw=none] (11.5, -0.5) rectangle (12.5,-16.5);
\draw[-] (8.5,-0.5) -- (8.5,-16.5);
\node (G1) at (6.5, 0) {Game 2};
\node (r1c1) at (1,-1) {$b2+$};
\node (r1c2) at (2,-1) {$a2$};
\node (r1c3) at (3,-1) {$a1$};
\node (r1c4) at (4,-1) {$b1$};
\node (r1c5) at (5,-1) {$a1+1$};
\node (r1c6) at (6,-1) {$a1$};
\node (r1c7) at (7,-1) {$(S)c1$};
\node (r1c8) at (8,-1) {$a3-1$};
\node (r2c7) at (7,-2) {$b3-11$};
\node (r2c8) at (8,-2) {$b3$};
\node (r3c5) at (5,-3) {$b3<(2)$};
\node (r3c6) at (6,-3) {$b2$};
\node (r4c3) at (3,-4) {$b3<1$};
\node (r4c4) at (4,-4) {$b2$};
\node (r5c3) at (3,-5) {$b3<2$};
\node (r5c4) at (4,-5) {$a2+$};
\node (r6c3) at (3,-6) {$Sa1$};
\node (r6c4) at (4,-6) {$c2$};
\node (r6c5) at (5,-6) {$b3-2$};
\node (r6c6) at (6,-6) {$a2>1$};
\node (r7c5) at (5,-7) {$a1+1$};
\node (r7c6) at (6,-7) {$c3$};
\node (r7c7) at (7,-7) {$(S)c1$};
\node (r7c8) at (8,-7) {$a3>1$};
\node (r8c5) at (5,-8) {$(S)b2$};
\node (r8c6) at (6,-8) {$a3>1$};
\node (r8c7) at (7,-8) {$b2>1$};
\node (r8c8) at (8,-8) {$a3$};
\node (r9c7) at (7,-9) {$a1+1$};
\node (r9c8) at (8,-9) {$c3$};
\node (r10c7) at (7,-10) {$b2<1$};
\node (r10c8) at (8,-10) {$c3$};
\node (r11c1) at (1,-11) {$c3$};
\node (r11c2) at (2,-11) {$a1$};
\node (r11c3) at (3,-11) {$b2<1$};
\node (r11c4) at (4,-11) {$b1$};
\node (r12c3) at (3,-12) {$Sa2$};
\node (r12c4) at (4,-12) {$c1$};
\node (r12c5) at (5,-12) {$(S)b1$};
\node (r12c6) at (6,-12) {$c2$};
\node (r13c5) at (5,-13) {$a2-1$};
\node (r13c6) at (6,-13) {$b3>1$};
\node (r14c5) at (5,-14) {$b2-1$};
\node (r14c6) at (6,-14) {$b3>1$};
\node (r15c3) at (3,-15) {$a2$};
\node (r15c4) at (4,-15) {$a3-1$};
\node (r15c5) at (5,-15) {$Sa3$};
\node (r15c6) at (6,-15) {$b1$};
\node (r16c1) at (1,-16) {$Sc3$};
\node (next) at (3, -16) {See second column};
\node (r16c2) at (9,-1) {$b3-1$};
\node (r16c3) at (10,-1) {$c3<1$};
\node (r16c4) at (11,-1) {$a2$};
\node (r17c3) at (10,-2) {$(S)b3$};
\node (r17c4) at (11,-2) {$a2$};
\node (r18c3) at (10,-3) {$(S)b1$};
\node (r18c4) at (11,-3) {$a2$};
\node (r19c3) at (10,-4) {$Sc1$};
\node (r19c4) at (11,-4) {$a2$};
\node (r20c3) at (10,-5) {$Sa1$};
\node (r20c4) at (11,-5) {$a2$};
\node (r21c3) at (10,-6) {$(S)c2$};
\node (r21c4) at (11,-6) {$a2$};
\node (r22c3) at (10,-7) {$c3-1$};
\node (r22c4) at (11,-7) {$a2$};
\node (r23c3) at (10,-8) {$(S)a2$};
\node (r23c4) at (11,-8) {$b1$};
\node (r23c5) at (12,-8) {$Sb3$};
\node (r23c6) at (13,-8) {$a1$};
\node (r24c5) at (12,-9) {$c3<1$};
\node (r24c6) at (13,-9) {$a1$};
\node (r25c5) at (12,-10) {$a2>1$};
\node (r25c6) at (13,-10) {$a1$};
\end{tikzpicture}
\end{center}

\begin{center}
\begin{tikzpicture}
\filldraw[fill=lightgray, draw=none] (0.5, -0.5) rectangle (1.5,-12.5);
\filldraw[fill=lightgray, draw=none] (2.5, -0.5) rectangle (3.5,-12.5);
\filldraw[fill=lightgray, draw=none] (4.5, -0.5) rectangle (5.5,-12.5);
\filldraw[fill=lightgray, draw=none] (6.5, -0.5) rectangle (7.5,-12.5);
\node (G1) at (4.5, 0) {Game 3};
\node (r1c1) at (1,-1) {$Sa3$};
\node (r1c2) at (2,-1) {$a2$};
\node (r1c3) at (3,-1) {$Sb2$};
\node (r1c4) at (4,-1) {$c2$};
\node (r1c5) at (5,-1) {$b1>1$};
\node (r1c6) at (6,-1) {$b1$};
\node (r1c7) at (7,-1) {$c1+1$};
\node (r1c8) at (8,-1) {$a1$};
\node (r2c3) at (3,-2) {$a3>1$};
\node (r2c4) at (4,-2) {$c2$};
\node (r3c3) at (3,-3) {$b1+1$};
\node (r3c4) at (4,-3) {$c2$};
\node (r4c3) at (3,-4) {$b2$};
\node (r4c4) at (4,-4) {$c2$};
\node (r5c1) at (1,-5) {$a3$};
\node (r5c2) at (2,-5) {$a2$};
\node (r5c3) at (3,-5) {$a3>1$};
\node (r5c4) at (4,-5) {$c3<1$};
\node (r5c5) at (5,-5) {$(S)c2$};
\node (r5c6) at (6,-5) {$a3$};
\node (r6c5) at (5,-6) {$Sb2$};
\node (r6c6) at (6,-6) {$b3<2$};
\node (r7c5) at (5,-7) {$b1+1$};
\node (r7c6) at (6,-7) {$b-12$};
\node (r8c5) at (5,-8) {$(S)c1$};
\node (r8c6) at (6,-8) {$b-12$};
\node (r9c5) at (5,-9) {$b2$};
\node (r9c6) at (6,-9) {$b-12$};
\node (r10c5) at (5,-10) {$(S)a1$};
\node (r10c6) at (6,-10) {$b-12$};
\node (r11c5) at (5,-11) {$(S)a3$};
\node (r11c6) at (6,-11) {$b-12$};
\node (r12c5) at (5,-12) {$(S)c3$};
\node (r12c6) at (6,-12) {$b-12$};
\end{tikzpicture}
\end{center}

Game $4$ is noteworthy because this position is avoidable unless White's first stone is in the center. We once again make mention that if a move is not considered it is because an equivalent move is.

\begin{center}
\begin{tikzpicture}
\filldraw[fill=lightgray, draw=none] (0.5, -0.5) rectangle (1.5,-10.5);
\filldraw[fill=lightgray, draw=none] (2.5, -0.5) rectangle (3.5,-10.5);
\filldraw[fill=lightgray, draw=none] (4.5, -0.5) rectangle (5.5,-10.5);
\filldraw[fill=lightgray, draw=none] (6.5, -0.5) rectangle (7.5,-10.5);
\filldraw[fill=lightgray, draw=none] (8.5, -0.5) rectangle (9.5,-10.5);
\node (G1) at (5.5, 0) {Game 4};
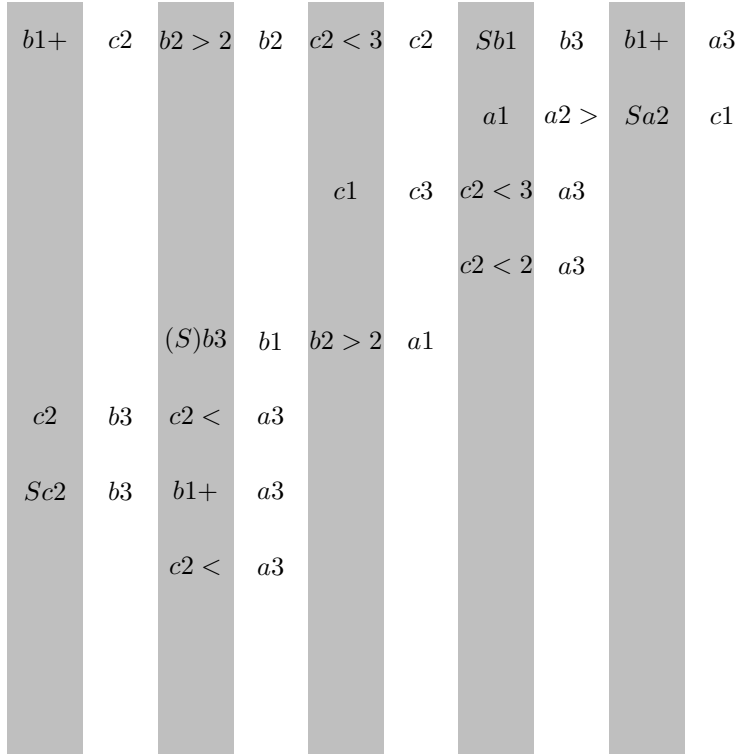
\node (r1c1) at (1,-1) {$b1+$};
\node (r6c1) at (1, -6) {$c2$};
\node (r7c1) at (1, -7) {$Sc2$};
\node (r1c2) at (2, -1) {$c2$};
\node (r6c2) at (2, -6) {$b3$};
\node (r7c2) at (2, -7) {$b3$};
\node (r1c3) at (3, -1) {$b2>2$};
\node (r5c3) at (3, -5) {$(S)b3$};
\node (r6c3) at (3, -6) {$c2<$};
\node (r7c3) at (3, -7) {$b1+$};
\node  (r8c3) at (3, -8) {$c2<$};
\node (r1c4) at (4, -1) {$b2$};
\node (r5c4) at (4, -5) {$b1$};
\node (r6c4) at (4, -6) {$a3$};
\node (r7c4) at (4, -7) {$a3$};
\node (r8c4) at (4, -8) {$a3$};
\node (r1c5) at (5, -1) {$c2<3$};
\node (r3c5) at (5, -3) {$c1$};
\node (r5c5) at (5, -5) {$b2>2$};
\node (r1c6) at (6, -1) {$c2$};
\node (r3c6) at (6, -3) {$c3$};
\node (r5c6) at (6, -5) {$a1$};
\node (r1c7) at (7, -1) {$Sb1$};
\node (r2c7) at (7, -2) {$a1$};
\node (r3c7) at (7, -3) {$c2<3$};
\node (r4c7) at (7, -4) {$c2<2$};
\node (r1c8) at (8, -1) {$b3$};
\node (r2c8) at (8, -2) {$a2>$};
\node (r3c8) at (8, -3) {$a3$};
\node (r4c8) at (8, -4) {$a3$};
\node (r1c9) at (9, -1) {$b1+$};
\node (r2c9) at (9, -2) {$Sa2$};
\node (r1c10) at (10, -1) {$a3$};
\node (r2c10) at (10, -2) {$c1$};
\end{tikzpicture}
\end{center}





\section*{Appendix}\label{A}
While \cite{takrules} claims there is a solution for $n=3$ none is provided. This was later shown by \cite{solved}. We verify this using \cite{takbot} and faithfully recreate their results here, as it hasn't appeared in the literature previously.

White begins by placing their opponents piece in a corner, without loss of generality $a1$. Black has $5$ possible unique responses, but after White's next move there are only three games to consider, up to symmetry. 

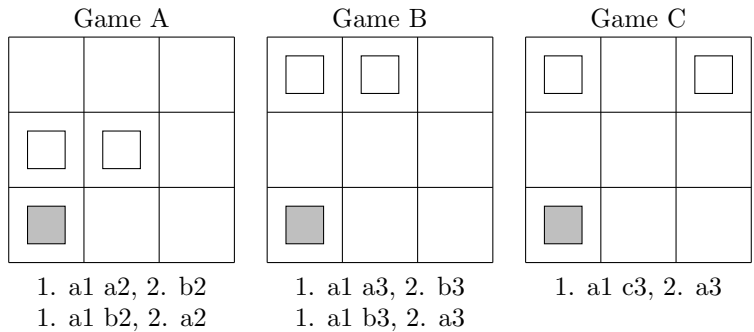
\begin{center}
\begin{tabular}{ccc}
Game A & Game B & Game C \\
\begin{tikzpicture}
    \foreach \x in {0,1,2,3} {
        \draw (\x,0) -- (\x,3); 
        \draw (0,\x) -- (3,\x); 
    }
    

    \filldraw[fill=lightgray, draw=black] (0.25, 0.25) rectangle (0.75, 0.75);
    \filldraw[fill=white, draw=black] (1.25, 1.25) rectangle (1.75, 1.75);
    \filldraw[fill=white, draw=black] (0.25, 1.25) rectangle (0.75, 1.75);
\end{tikzpicture}
&
\begin{tikzpicture}
    \foreach \x in {0,1,2,3} {
        \draw (\x,0) -- (\x,3); 
        \draw (0,\x) -- (3,\x); 
    }
    

    \filldraw[fill=lightgray, draw=black] (0.25, 0.25) rectangle (0.75, 0.75);
    \filldraw[fill=white, draw=black] (1.25, 2.25) rectangle (1.75, 2.75);
    \filldraw[fill=white, draw=black] (0.25, 2.25) rectangle (0.75, 2.75);
\end{tikzpicture}
&
\begin{tikzpicture}
    \foreach \x in {0,1,2,3} {
        \draw (\x,0) -- (\x,3); 
        \draw (0,\x) -- (3,\x); 
    }
    

    \filldraw[fill=lightgray, draw=black] (0.25, 0.25) rectangle (0.75, 0.75);
    \filldraw[fill=white, draw=black] (2.25, 2.25) rectangle (2.75, 2.75);
    \filldraw[fill=white, draw=black] (0.25, 2.25) rectangle (0.75, 2.75);
\end{tikzpicture}\\
1. a1 a2, 2. b2 & 1. a1 a3, 2. b3 & 1. a1 c3, 2. a3\\
1. a1 b2, 2. a2 & 1. a1 b3, 2. a3 &
\end{tabular}
\end{center}

We consider each game separately. Any time a move is not considered, it is because it would allow white to win on the next move, or because it brings the board to an equivalent state as one which is considered. 
Due to its simplicity, we've elected to make Game B a tree. We continue to shade the rows to denote which player makes which move. After Game $C$ is reached Black has two moves: $b3$, $Sb3$.

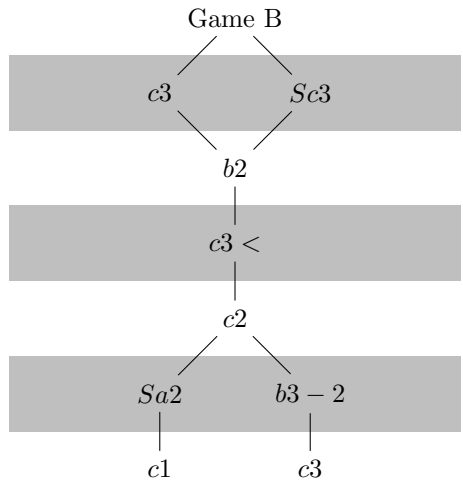
\begin{center}
\begin{tikzpicture}
\node (G1) at (6, 1) {Game A};
\filldraw[fill=lightgray, draw=none] (-0.5, 0.5) rectangle (0.5,-17.5);
\filldraw[fill=lightgray, draw=none] (1.5, 0.5) rectangle (2.5,-17.5);
\filldraw[fill=lightgray, draw=none] (3.5, 0.5) rectangle (4.5,-17);
\filldraw[fill=lightgray, draw=none] (7.5, 0.5) rectangle (8.5,-20.5);
\filldraw[fill=lightgray, draw=none] (9.5, 0.5) rectangle (10.5,-20.5);
\filldraw[fill=lightgray, draw=none] (11.5, 0.5) rectangle (12.5,-20.5);
\draw[-] (6, 0.5) -- (6, -20.5);
\node (r1c1) at (0,0) {$c2$};
\node (r1c2) at (1,0) {$b3$};
\node (r1c3) at (2,0) {$c2<1$};
\node (r1c4) at (3,0) {$a2>1$};
\node (r1c5) at (4,0) {$Sb1$};
\node (r1c6) at (5,0) {$a3$};
\node (r2c1) at (0,-0.75) {$Sc2$};
\node (r2c2) at (1,-0.75) {$a2-1$};
\node (r2c3) at (2,-0.75) {$(S)a2$};
\node (r2c4) at (3,-0.75) {$b1$};
\node (r3c3) at (2,-1.5) {$(S)a3$};
\node (r3c4) at (3,-1.5) {$b1$};
\node (r4c3) at (2,-2.25) {$(S)b1$};
\node (r4c4) at (3,-2.25) {$a3$};
\node (r4c5) at (4,-2.25) {$Sa2$};
\node (r4c6) at (5,-2.25) {$b3$};
\node (r5c5) at (4,-3) {$b1<1$};
\node (r5c6) at (5,-3) {$b3$};
\node (r6c3) at (2,-3.75) {$b3$};
\node (r6c4) at (3,-3.75) {$a3$};
\node (r6c5) at (4,-4.5) {$Sa2$};
\node (r6c6) at (5,-4.5) {$b3$};
\node (r7c5) at (4,-5.25) {$b1<1$};
\node (r7c6) at (5,-5.25) {$b3$};
\node (r8c3) at (2,-6) {$b3$};
\node (r8c4) at (3,-6) {$a3$};
\node (r8c5) at (4,-6) {$Sa2$};
\node (r8c6) at (5,-6) {$b1$};
\node (r9c5) at (4,-6.75) {$b3<1$};
\node (r9c6) at (5,-6.75) {$b1$};
\node (r10c3) at (2,-7.5) {$Sb3$};
\node (r10c4) at (3,-7.5) {$a2$};
\node (r10c5) at (4,-7.5) {$(S)a3$};
\node (r10c6) at (5,-7.5) {$b1$};
\node (r11c3) at (2,-8.25) {$(S)c1$};
\node (r11c4) at (3,-8.25) {$a2$};
\node (r12c3) at (2,-9) {$c2<1$};
\node (r12c4) at (3, -9) {$c3$};
\node (r12c5) at (4, -9) {$(S)a2$};
\node (r12c6) at (5, -9) {$c1$};
\node (r13c5) at (4, -9.75) {$(S)a3$};
\node (r13c6) at (5, -9.75) {$c1$};
\node (r14c5) at (4, -10.5) {$b2<1$};
\node (r14c6) at (5, -10.5) {$b1$};
\node (r15c5) at (4, -11.25) {$b2<2$};
\node (r15c6) at (5, -11.25) {$c1$};
\node (r16c5) at (4, -12) {$b2>1$};
\node (r16c6) at (5, -12) {$a2$};
\node (r17c5) at (4, -12.75) {$b2>2$};
\node (r17c6) at (5, -12.75) {$a3$};
\node (r18c5) at (4, -13.5) {$(S)b3$};
\node (r18c6) at (5, -13.5) {$c1$};
\node (r19c3) at (2,-14.25) {$c2\pm1$};
\node (r19c4) at (3,-14.25) {$a2$};
\node (r20c3) at (2,-15) {$(S)c3$};
\node (r20c4) at (3, -15) {$a2$};
\node (r21c1) at (0, -15.75) {$a1+1$};
\node (r21c2) at (1, -15.75) {$b3$};
\node (r21c3) at (2, -15.75) {$a2>1$};
\node (r21c4) at (3, -15.75) {$a3$};
\node (r22c3) at (2, -16.5) {$(S)b1$};
\node (r22c4) at (3, -16.5) {$a3$};
\node (r22c5) at (4,-16.5) {$a2+2$};
\node (r22c6) at (5,-16.5) {$b3<1$};
\node (r23c3) at (2, -17.25) {$a2>2$};
\node[right] (move) at (3, -17.25) {See second column};
\node (r0c7) at (7,0) {$a3$};
\node (r0c8) at (8,0) {$b2+1$};
\node (r0c9) at (9,0) {$a2$};
\node (r1c8) at (8,-0.75) {$b2+2$};
\node (r1c9) at (9,-0.75) {$a2$};
\node (r2c8) at (8,-1.5) {$b2+3$};
\node (r2c9) at (9,-1.5) {$b2$};
\node (r2c10) at (10,-1.5) {$(S)a1$};
\node (r2c11) at (11,-1.5) {$b2+1$};
\node (r3c10) at (10,-2.25) {$(S)a2$};
\node (r3c11) at (11,-2.25) {$b2+1$};
\node (r4c10) at (10,-3) {$(S)b1$};
\node (r4c11) at (11,-3) {$a2$};
\node (r5c10) at (10,-3.75) {$(S)c1$};
\node (r5c11) at (11,-3.75) {$a2$};
\node (r6c10) at (10,-4.5) {$(S)c2$};
\node (r6c11) at (11,-4.5) {$a2$};
\node (r7c10) at (10,-5.25) {$(S)c3$};
\node (r7c11) at (11,-5.25) {$a2$};
\node (r8c8) at (8,-6) {$c3$};
\node (r8c9) at (9,-6) {$a1$};
\node (r8c10) at (10,-6) {$a2$};
\node (r8c11) at (11,-6) {$a3-1$};
\node (r8c12) at (12,-6) {$Sa3$};
\node (r8c13) at (13,-6) {$b1$};
\node (r9c12) at (12,-6.75) {$b2<3$};
\node (r9c13) at (13,-6.75) {$b1$};
\node (r10c10) at (10,-7.5) {$Sa2$};
\node (r10c11) at (11,-7.5) {$c1$};
\node (r10c12) at (12,-7.5) {$a2-1$};
\node (r10c13) at (13,-7.5) {$b3>1$};
\node (r11c12) at (12,-8.25) {$b2-3$};
\node (r11c13) at (13,-8.25) {$a1>1$};
\node (r12c10) at (10,-9) {$b2<3$};
\node (r12c11) at (11,-9) {$a1+1$};
\node (r13c8) at (8,-9.75) {$Sc3$};
\node (r13c9) at (9,-9.75) {$b1$};
\node (r13c10) at (10,-9.75) {$(S)a1$};
\node (r13c11) at (11,-9.75) {$b3-1$};
\node (r13c12) at (12,-9.75) {$a1>1$};
\node (r13c13) at (13,-9.75) {$a2$};
\node (r14c10) at (10,-10.5) {$a2$};
\node (r14c11) at (11,-10.5) {$a3-1$};
\node (r14c12) at (12,-10.5) {$b2-3$};
\node (r14c13) at (13,-10.5) {$a1$};
\node (r15c12) at (12,-11.25) {$b2+3$};
\node (r15c13) at (13,-11.25) {$a1$};
\node (r16c12) at (12,-12) {$c3<1$};
\node (r16c13) at (13,-12) {$a1$};
\node (r17c10) at (10,-12.75) {$Sa2$};
\node (r17c11) at (11,-12.75) {$b3-1$};
\node (r17c12) at (12,-12.75) {$a2>1$};
\node (r17c13) at (13,-12.75) {$a1$};
\node (r18c12) at (12,-13.5) {$Sb3$};
\node (r18c13) at (13,-13.5) {$a1$};
\node (r19c12) at (12,-14.25) {$c3<1$};
\node (r19c13) at (13,-14.25) {$a1$};
\node (r20c10) at (10,-15) {$b2-3$};
\node (r20c11) at (11,-15) {$b2$};
\node (r20c12) at (12,-15) {$b1+3$};
\node (r20c13) at (13,-15) {$a1$};
\node (r21c12) at (12,-15.75) {$(S)c2$};
\node (r21c13) at (13,-15.75) {$a1$};
\node (r22c12) at (12,-16.5) {$c3<1$};
\node (r22c13) at (13,-16.5) {$a2$};
\node (r23c10) at (10,-17.25) {$b2+3$};
\node (r23c11) at (11,-17.25) {$a1$};
\node (r24c10) at (10,-18) {$(S)c1$};
\node (r24c11) at (11,-18) {$a1$};
\node (r25c10) at (10,-18.75) {$c2$};
\node (r25c11) at (11,-18.75) {$a2$};
\node (r26c10) at (10,-19.5) {$Sc2$};
\node (r26c11) at (11,-19.5) {$a1$};
\node (r27c10) at (10,-20.25) {$c3<1$};
\node (r27c11) at (11,-20.25) {$a1$};
\end{tikzpicture}
\end{center}

\begin{center}
\begin{tikzpicture}
\filldraw[fill=lightgray, draw=none] (-3, 4.5) rectangle (3,5.5);
\filldraw[fill=lightgray, draw=none] (-3, 2.5) rectangle (3,3.5);
\filldraw[fill=lightgray, draw=none] (-3, 0.5) rectangle (3,1.5);
\node (B) at (0, 6) {Game B};
\node (c3) at (-1, 5) {$c3$};
\node (Sc3) at (1, 5) {$Sc3$};
\node (b2) at (0, 4) {$b2$};
\node (2c3) at (0,3) {$c3<$};
\node (c2) at (0,2) {$c2$};
\node (Sa2) at (-1,1) {$Sa2$};
\node (b3) at (1,1) {$b3-2$};
\node (c1) at (-1, 0) {$c1$};
\node (3c3) at (1,0) {$c3$};
\draw[-] (B) -- (c3) -- (b2) -- (Sc3)--(B);
\draw[-] (b2) -- (2c3) -- (c2) -- (Sa2) -- (c1);
\draw[-] (c2) -- (b3) -- (3c3);
\end{tikzpicture}
\end{center}

\setlength{\tabcolsep}{1pt}
\begin{center}
\begin{table}[h]
\begin{tabular}{>{\columncolor[gray]{0.8}}cc>{\columncolor[gray]{0.8}}cc>{\columncolor[gray]{0.8}}cc>{\columncolor[gray]{0.8}}cc>{\columncolor[gray]{0.8}}cc|>{\columncolor[gray]{0.8}}cc>{\columncolor[gray]{0.8}}cc>{\columncolor[gray]{0.8}}cc}
&&&&&&&&Game C&&&&&&&\\
b3 &  c1 &  b3\textgreater{}1 &  b2 &  a1\textgreater{}1 &  a2 &   &   &   &   &  Sb2 &  c2 &  a1\textgreater{}1 &  a2 &   & \\
 &   &   &   &  a2 &  a3-1 &  a1+1 &  b1 &   &   &   &   &  b2\textgreater{}1 &  a2 &  a1+1 &  b1 \\
 &   &   &   &  Sa2 &  b3 &  a2\textgreater{}1 &  b3\textgreater{}1 &   &   &   &   &   &   &  c2\textless{}2 &  a2-1 \\
 &   &   &   &   &   &  c3\textless{}2 &  b2+1 &   &   &   &   &   &   &  c2-2 &  b2 \\
 &   &  &   &  b3 &  a3\textgreater{}1 &  c3\textless{}2 &  c2 &   &   &   &   &   &   &  c3\textless{}(2) &  a2-1 \\
 &   &   &   &  Sb3 &  a2 &  a1+1 &  b1 &   &   &  c2 &  b1 &  a1\textgreater{}1 &  a2 &   &   \\
 &   &   &   &   &   &  b3-1 &  a2-1 &   &   &   &   &  c2\textless{}1 &  b1\textless{}1 &   &   \\
 &   &   &   &  c3\textless{}2 &  a3\textgreater{}1 &   &   &   &   &   &   &  c3\textless{}1 &  a2 &   &   \\
 &   &  c2 &  b2 &  b3\textgreater{}1 &  a2 &  a1+1 &  b1 &   &   &   &   &  c3\textless{}2 &  a2 &   &   \\
 &   &   &   &   &   &  c2\textless{}1 &  a2-1 &   &   &  Sc2 &  a2 &  a1+1 &  b1 &   &   \\
 &   &  Sc2 &  a2 &  a1+1 &  b1 &  a2-2 &  a2 &   &   &   &   &  c2\textless{}1 &  a2-1 &   &   \\
 &   &   &   &   &   &  c2-1 &  a3\textgreater{}1 &   &   &   &   &  c2-1 &  b2 &   &   \\
 &   &   &   &  b3\textless{}1 &  a2+1 &  Sb3 &  a3-12 &   &   &   &   &  c3\textless{}1 &  a2-1 &   &   \\
 &   &   &   &   &   &  c2+1 &  a3-12 &   &   &   &   &  c3\textless{}2 &  a2-1 &   &   \\
 &   &   &   &  b3\textgreater{}1 &  a2-1 &   &   &   &   &  c3\textless{}11 &  c2 &   &   &   &   \\
 &   &   &   &  c2-1 &  a3\textgreater{}1 &   &   &   &   &  c3\textless{}2 &  a2 &  a1+1 &  a3-1 &  a1 &  a2-2 \\
 &   &   &   &  c2+1 &  a2-1 &   &   &   &   &   &   &   &   &  Sa1 &  a2\textgreater{}12 \\
Sb3 &  c1 &  b3\textgreater{}1 &  b3 &  \multicolumn{5}{l}{Second column contains more moves} &   &   &   &   &   &  a3 &  a1 \\
 &   &   &   &  a1\textgreater{}1 &  a2 &  a1+1 &  c1\textless{}1 &   &   &   &   &   &   &  Sa3 &  a2\textgreater{}12 \\
 &   &   &   &   &   &  Sa1 &  c1+1 &   &  &   &   &   &   &  (S)b1 &  a2\textgreater{}12 \\
 &   &   &   &   &   &  b1\textless{}1 &  a2-1 &   &   &   &   &   &   &  (S)b2 &  a1 \\
 &   &   &   &  a1+1 &  b1 &   &   &   &   &   &   &   &   &  b3\textless{}3 &  a2\textgreater{}12 \\
 &   &   &   &  a2 &  b1 &  a1\textgreater{}1 &  c2 &  a2\textgreater{}1 &  a2 &   &   &   &   &  b3-(3) &  a1 \\
 &   &   &   &   &   &   &   &  b1+2 &  b1 &   &   &   &   &  b3\textgreater{}3 &  a1 \\
 &   &   &   &   &   &  a2\textgreater{}1 &  b1\textless{}1 &   &   &   &   &   &   &  (S)c2 &  a1 \\
 &   &   &   &   &   &  c3\textless{}2 &  a3-1 &  a1\textgreater{}1 &  c2 &   &   &   &   &  (S)c3 &  a1 \\
 &   &   &   &  Sa2 &  b1 &  a1\textgreater{}1 &  c2 &  a2\textgreater{}1 &  a2 &   &   &  (S)b1 &  c2 &   &   \\
 &   &   &   &   &   &  a2\textgreater{}1 &  b1\textless{}1 &   &   &   &   &  (S)b2 &  a2-1 &   &   \\
 &   &   &   &   &   &  c3\textless{}2 &  c2 &  b3\textgreater{}3 &  b3 &   &   &  b3\textless{}2 &  b2 &   &   \\
 &   &   &   &  b1 &  c2 &  b1+1 &  a2-1 &  a1+1 &  b1 &   &   &  b3\textless{}3 &  c2 &   &   \\
 &   &   &   &   &   &   &   &  b2\textgreater{}1 &  a2-1 &   &   &  b3-3 &  a2-1 &   &   \\
 &   &   &   &  Sb1 &  a2 &  c1+1 &  c2 &  a2\textgreater{}2 &  a2 &   &   &  b3\textgreater{}(3) &  a2-1 &   &   \\
 &   &   &   &   &   &   &   &  b1+1 &  b1 &   &   &  (S)c2 &  a2-1 &   &   \\
 &   &   &   &   &   &  b1\textless{}1 &  b2 &   &   &   &   &  (S)c3 &  a2-1 &   &   \\
 &   &   &   &   &   &  b1\textgreater{}1 &  b2 &   &   &  b2 &  a2 &  a1+1 &  a3-1 &  b2\textless{}1 &  b1 \\
 &   &   &   &   &   &  b1+1 &  a2-1 &   &   &   &   &   &   &  b2+1 &  a1 \\
 &   &   &   &   &   &  c3\textless{}(2) &  c2 &   &   &   &   &   &   &  Sc2 &  a1 \\
 &   &  c2 &  b1 &  a1\textgreater{}1 &  a2 &  b1\textless{}2 &  a2-1 &   &   &   &   &   &   &  c3\textless{}2 &  a1 \\
 &   &   &   &   &   &  b3\textless{}1 &  c1+1 &   &   &   &   &   &   &  c3-2 &  a1 \\
 &   &   &   &  b3\textless{}1 &  b2 &   &   &   &   &   &   &  b2\textless{}1 &  a3-1 &  a1+1 &  b1 \\
 &   &   &   &  b3\textgreater{}1 &  a2 &   &   &   &   &   &   &   &   &  Sb2 &  a2-2 \\
 &   &   &   &  c2-1 &  b1\textgreater{}1 &  b3\textgreater{}1 &  c1\textless{}12 &   &   &   &   &   &   &  Sc2 &  a2-2 \\
 &   &   &   &   &   &  Sc2 &  c12\textless{}12 &   &   &   &   &   &   &  c3\textless{}2 &  a2-2 \\
 &   &   &   &  c2+1 &  a2 &   &   &   &   &   &   &   &   &  c3-2 &  a2-2 \\
 &   &  Sc2 &  b2 &  a1\textgreater{}1 &  a2 &   &   &   &   &   &   &  b2-1 &  c1\textless{}1 &   &   \\
 &   &   &   &  b3\textless{}1 &  a2 &  a1+1 &  b1 &   &   &   &   &  b2\textgreater{}1 &  a2-1 &   &   \\
 &   &   &   &   &   &  a3-2 &  b3 &   &   &   &   &  b2+1 &  a2-1 &   &   \\
 &   &   &   &  b3\textgreater{}1 &  a2 &  a1+1 &  b1 &   &   &   &   &  c3\textless{}2 &  a2-1 &   &   \\
 &   &   &   &   &   &  c2\textless{}1 &  a2-1 &   &   &   &   &   &   &   &
\end{tabular}
\end{table}
\end{center}


\begin{thebibliography}{8}
\bibitem{solved} Joshua Achiam. {\it Solution to 3x3 Tak: Opening by Putting Black in a Corner}. https://github.com/jachiam/tak-ai/blob/master/3x3/black-in-corner/black-in-corner-soln.md
\bibitem{math} Boris Alexeev, Paul Ellis, Michael Richter, and Thotsapron Aek Thanatipanonda. The Penults of TAK: Adventures in Impartial, Normal-Play Positional Games, {\it INTEGERS} {\bf 25} (2025), Paper No. G1, 18
\bibitem{takbot} Nelson Elhage. {\it Taktician - A Tak Bot}. https://github.com/nelhage/taktician
\bibitem{takrules} James Ernest, and Patrick Rothfuss, {\it Tak Companion Book}, Cheapass Games, LLC, 2016.
\bibitem{master1} Bill Leighton, {\it Mastering Tak: Level I: A Foundation for Success}, CreateSpace Independent Publishing Platform, 2017
\bibitem{master2} Bill Leighton, {\it Mastering Tak: Level II: The road to greatness}, Independently published, 2023
\bibitem{wiseman} Patrick Rothfuss, {\it Wise Mans Fear}, DAW Books, 2011
\bibitem{oeis} J. A. Sloane. {\it The On-line Encyclopedia of Integer Sequences}. www.oeis.org
\end{thebibliography}
\end{document}